\documentclass[11pt]{article}

\usepackage{amssymb, amsmath, amsthm, graphicx}
\usepackage[left=1in,top=1in,right=1in]{geometry}

\usepackage{subfigure}

      \theoremstyle{plain}
      \newtheorem{theorem}{Theorem}[section]
      \newtheorem{lemma}[theorem]{Lemma}
            \newtheorem{problem}[theorem]{Problem}
            
      \newtheorem{corollary}[theorem]{Corollary}

      \theoremstyle{definition}

      \theoremstyle{remark}

\def\ocn{\mbox{\rm odd-cr}}

\title{Density theorems for intersection graphs of $t$-monotone curves}

\author{Andrew Suk\thanks{Massachusetts Institute of Technology, Cambridge. Email: {\tt asuk@math.mit.edu}.  Research funded by an NSF Postdoctoral Fellowship. }}

\begin{document}

\maketitle

\begin{abstract}
A curve $\gamma$ in the plane is $t$-monotone if its interior has at most $t-1$ vertical tangent points.  A family of $t$-monotone curves $F$ is \emph{simple} if any two members intersect at most once.  It is shown that if $F$ is a simple family of $n$ $t$-monotone curves with at least $\epsilon n^2$ intersecting pairs (disjoint pairs), then there exists two subfamilies $F_1,F_2 \subset F$ of size $\delta n$ each, such that every curve in $F_1$ intersects (is disjoint to) every curve in $F_2$, where $\delta$ depends only on $\epsilon$.  We apply these results to find pairwise disjoint edges in simple topological graphs with $t$-monotone edges.

\end{abstract}

\section{Introduction}

Given a collection of objects $C$ in the plane, the \emph{intersection graph} $G(C)$ has vertex set $C$ and two objects are adjacent if and only if they have a nonempty intersection.  Intersection graphs have recently received a lot of attention due to its applications in VLSI design \cite{vlsi}, map labeling \cite{map,imai}, graph drawing \cite{fox,fulek,suk,toth}, and graph theory \cite{kozik}.  Over the past several decades, many researchers have shown that intersection graphs of certain geometric objects, such as of segments, chords of a circle, axis parallel rectangles, etc.,  have very strong properties (see \cite{gyarfas,asplund,suk2,fox}).  The aim of this paper is to establish density-type theorems for intersection graphs of $t$-monotone curves in the plane.

A \emph{curve} in the plane is the image of a smooth injective function $f:I \rightarrow \mathbb{R}^2$, whose domain is a closed interval $I\subset \mathbb{R}$.  For integer $t\geq 1$, a curve $\gamma$ in the plane is \emph{$t$-monotone},\footnote{A 1-monotone curve is often referred to as \emph{$x$-monotone}.  Every $t$-monotone curve can be decomposed into $t$ 1-monotone curves.} if its interior has at most $t-1$ vertical tangent points.  In 2001, Pach and Solymosi proved the following two theorems.

 \begin{theorem}[\cite{solymosi}]
 \label{first}
 Let $F$ be a family of $n$ segments in the plane with at least $\epsilon n^2$ intersecting pairs.  Then there exists a $\delta$ that depends only on $\epsilon$, and two subfamilies $F_1,F_2 \subset F$, such that $|F_1|,|F_2|\geq \delta n$, and every segment in $F_1$ intersects every segment in $F_2$.
 \end{theorem}

 \begin{theorem}[\cite{solymosi}]
  \label{second}
  Let $F$ be a family of $n$ segments in the plane with at least $\epsilon n^2$ disjoint pairs.  Then there exists a $\delta$ that depends only on $\epsilon$, and two subfamilies $F_1,F_2 \subset F$, such that $|F_1|,|F_2|\geq \delta n$, and every segment in $F_1$ is disjoint to every segment in $F_2$.
 \end{theorem}

 These theorems were later generalized by Alon et al. to semi-algebraic sets in $\mathbb{R}^d$ \cite{alon}, and by Basu \cite{basu} to definable sets belonging to some fixed definable family of sets in an o-minimal structure.  In all three papers \cite{solymosi}, \cite{alon}, \cite{basu}, the authors only considered geometric objects with bounded or fixed \emph{description complexity}, that is, each geometric object can be encoded as a point in $\mathbb{R}^q$ where $q = q(d)$.  Previously, there were no known generalizations of Theorem \ref{second} to geometric objects with large complexity.

Our main result generalizes Theorems \ref{first} and \ref{second} to $t$-monotone curves, which can have arbitrarily large complexity.  We say that a family of curves $F$ is \emph{simple} if any two members intersect at most once and no two curves are tangent, i.e., if two curves have a common interior point, they must properly cross at that point.

 \begin{theorem}
\label{main1}
Let $F$ be a simple family of $n$ $t$-monotone curves in the plane with at least $\epsilon n^2$ intersecting pairs.  Then there exists a constant $c_t$ that depends only on $t$, and two subfamilies $F_1,F_2 \subset F$, such that $|F_1|,|F_2|\geq \epsilon^{c_t} n$, and every curve in $F_1$ intersects every curve in $F_2$.

\end{theorem}

\begin{theorem}
\label{main2}
Let $F$ be a simple family of $n$ $t$-monotone curves in the plane with at least $\epsilon n^2$ disjoint pairs. Then there exists a constant $c_t$ that depends only on $t$, and two subfamilies $F_1,F_2 \subset F$, such that $|F_1|,|F_2|\geq \epsilon^{c_t} n$, and every curve in $F_1$ is disjoint to every curve in $F_2$.
\end{theorem}

\noindent Interestingly, Theorem \ref{main2} does not hold if one drops the simple condition.  That is, there exists a family $F$ of $n$ 1-monotone curves in the plane, with at least $n^2/4$ disjoint pairs, such that for any two subsets $F_1,F_2\subset F$ of size $\Omega(\log n)$ each, there exists a curve in $F_1$ that crosses a curve in $F_2$ \cite{pachtalk}.

While Theorem \ref{main2} is a new result, Theorem \ref{main1} is a special case of a theorem due to Fox, Pach, and T\'oth.  In \cite{fox}, Fox, Pach, and T\'oth generalized Theorem \ref{first} to families of curves in the plane with the property that any two curves intersect at most a constant number of times.  Let us remark that our proof is conceptually simpler.

 \subsection{Applications to topological graphs}

A \emph{topological graph} is a graph drawn in the plane such that its vertices are represented by points
 and its edges are represented by nonself-intersecting arcs connecting the corresponding points. The arcs are
allowed to intersect, but they may not intersect vertices except for
their endpoints.  Furthermore, no two edges are tangent, i.e., if two edges share an interior point, then they must properly cross at that point in common.  A topological graph is \emph{simple} if every pair of its edges intersect at most once.  Two edges of a topological graph \emph{cross} if their interiors share a point, and are \emph{disjoint} if they neither share a common vertex nor cross.

Over 40 years ago, Conway asked what is the maximum number of edges in a \emph{thrackle}, that is, a simple topological graph with no two disjoint edges.  He conjectured that every $n$-vertex thrackle has at most $n$ edges.  Lov\'asz, Pach, and Szegedy \cite{lovasz} were the first to establish a linear bound, proving that all such graphs have at most $2n$ edges.  Despite recent improvements by Cairns and Nikolayevsky \cite{cairns1} and Fulek and Pach \cite{fulek}, this conjecture is still open.  In the special case that the edges are drawn as 1-monotone curves, Pach and Sterling settled Conway's conjecture in the affirmative \cite{ster}.

Determining the maximum number of edges in a simple topological graph with no $k$ pairwise disjoint edges, seems to be a difficult task.  Pach and T\'oth \cite{pach2} showed that every simple topological graph with no $k$ pairwise disjoint edges has at most $O(n\log^{4k-8}n)$ edges.  They conjectured that for every fixed $k$, the number of edges in such graphs is at most $O(n)$.  A linear bound was obtained by Pach and T\"or\"ocsik \cite{pach}, in the special case that the edges are drawn as $1$-monotone curves (see also \cite{toth}).  As an application of Theorem \ref{main2}, we improve (for large $k$) the Pach and T\'oth bound, in the special case that the edges are drawn as $t$-monotone curves (where $t$ is independent of $n$).

\begin{theorem}
\label{disjoint}
Let $G = (V,E)$ be an $n$-vertex simple topological graph with edges drawn as $t$-monotone curves.  If $G$ does not contain $k$ pairwise disjoint edges, then $|E(G)| \leq n(\log n)^{c'_t \log k}$, where $c'_t$ depends only on $t$.
\end{theorem}

In 2009, Fox and Sudakov \cite{fox2} showed that all dense simple topological graphs have at least $\Omega(\log^{1 + \delta}n)$ pairwise disjoint edges, where $\delta \approx 1/40$.  As an immediate Corollary to Theorem \ref{disjoint}, we improve this lower bound (to nearly polynomial) in the special case that the edges are drawn as $t$-monotone curves

\begin{corollary}
Let $G = (V,E)$ be an $n$-vertex simple topological graph with edges drawn as $t$-monotone curves.  If $|E(G)| \geq \epsilon n^2$, then $G$ has at least $n^{\delta/\log\log n}$ pairwise disjoint edges, where $\delta$ depends only on $\epsilon$ and $t$.
\end{corollary}

We note that Suk recently showed that every complete $n$-vertex simple topological graph has at least $\Omega(n^{1/3})$ pairwise disjoint edges \cite{suk}.

\section{A two-color theorem}

In this section, we will prove the following two-color theorem.

\begin{theorem}
\label{key}
Given a family $B$ of $n$ blue $t$-monotone curves and a family $R$ of $n$ red $t$-monotone curves in the plane such that $B\cup R$ is simple, there exist a $c''_t>0$ that depends only on $t$, and subfamilies $B'\subset B$, $R'\subset R$, such that $|B'|,|R'| \geq n/c''_t$, and either each curve in $B'$ intersects every curve in $R'$, or each curve in $B'$ is disjoint to every curve in $R'$.
\end{theorem}

In what follows, we will prove a sequence of lemmas that will lead to the proof of Theorem \ref{key}.  First we need some definitions.  Let $F$ be a simple family of curves in the plane.  For $\gamma\in F$, the endpoint with the left (right) most $x$-coordinate we refer to as the \emph{left (right) endpoint} of $\gamma$.  By a slight perturbation, we can assume that all endpoints have unique $x$-coordinates, no curve has a vertical inflection point, no endpoint of one curve lies on another curve, and no three curves in $F$ have a nonempty intersection.  For any simply connected region $\Delta \subset \mathbb{R}^2$, we denote the boundary of $\Delta$ as $bd(\Delta)$.  For the rest of the paper, the term \emph{region} will always mean a simply connected region.

A point $q$ is called a \emph{critical point of} $\gamma \in F$, if $q$ is an endpoint of $\gamma$ or if $\gamma$ has a vertical tangent at $q$.  Now given a subset $S\subset F$, the \emph{vertical decomposition} of the arrangement of $S$ is constructed by subdividing the cells of the arrangement $\mathcal{A}(S) = \bigcup_{\gamma\in S}\gamma$ into trapezoid-like regions $\Delta_1,\Delta_2,...,\Delta_s$, by drawing a vertical line in both directions through every intersection point of a pair of curves and through every critical point of a curve in $S$ until it hits some element in $S$.  We let $\mathcal{T}(S)$ be the vertical decomposition of $S$, and for simplicity we will call the elements in $\mathcal{T}(S)$ \emph{trapezoids}.  See Figure \ref{trape}.  Note that a trapezoid may be unbounded.  Let $Reg = \bigcup_{S\subset F} \mathcal{T}(S)$ be the set of all trapezoids that can ever appear in the vertical decomposition for some $S\subset F$.  For each trapezoid $\Delta \in Reg$, let $D(\Delta)$ be the set of curves in $F$ that intersects the boundary of $\Delta$ but does not intersect the interior of $\Delta$.  One can easily check that $|D(\Delta)| \leq 4$ (see \cite{mat}).  Finally, we let $I(\Delta)$ denote the set of curves of $F$ intersecting the interior of $\Delta$.

\begin{figure}
\begin{center}
\includegraphics[width=150pt]{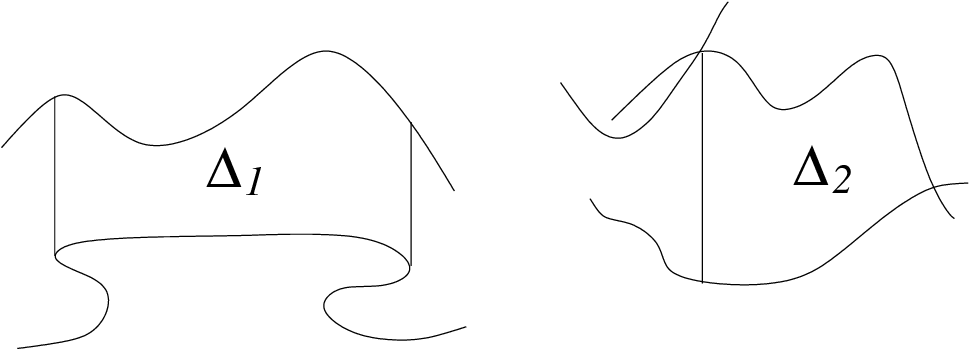}
  \caption{Trapezoids $\Delta_1$ and $\Delta_2$.  Notice that $|D(\Delta_1)| = 2$ and $|D(\Delta_2)| = 3$.}
  \label{trape}
 \end{center}
\end{figure}

\begin{lemma}
\label{sample}
Given a simple family $F$ of $n$ $t$-monotone curves and a set $P$ of $n$ points in the plane with no point in $P$ lying on any curve from $F$, there exists a constant $C_1$, subsets $F'\subset F$ and $P'\subset P$ of size $n/(C_1t\log^2 t)$ each, and a trapezoid $\Delta$, such that $P'\subset \Delta$ and no curve in $F'$ intersects the interior of $\Delta$.
\end{lemma}

\begin{proof} Let $S$ be a random subset of $F$ by selecting each curve in $F$ independently with probability $p=r/n$.  Notice that for any subset $S \subset F$, $|\mathcal{T}(S)| \leq C_2(|S|^2 + |S|t)$ for some constant $C_2$, since $F$ is simple.  An easy calculation shows that $\mathbb{E}[|S|t] = rt$ and $\mathbb{E}[|S|^2] \leq r^2 + r$, and thus

$$\mathbb{E}[|\mathcal{T}(S)|] \leq C_2(r^2 + r + rt ) \leq 3C_2r^2t.$$

For $\Delta \in Reg$, we let $p(\Delta)$ denote the probability that $\Delta$ appears in the vertical decomposition of $S$.  Since $\Delta$ appears if and only if all curves of $D(\Delta)$ are selected into $S$ and none of $I(\Delta)$ is selected, we have

$$p(\Delta)  = p^{|D(\Delta)|} (1 - p)^{|I(\Delta)|}.$$

We call a trapezoid $\Delta \in Reg$ \emph{bad} if $|I(\Delta)| \geq n/2$, otherwise it is \emph{good}.  Let $Bad = \{\Delta \in Reg : |I(\Delta)| \geq n/2\}$, and let $X$ denote the number of bad trapezoids in the vertical decomposition of $S$.  Since $|D(\Delta)| \leq 4$, and any four curves define at most $20t$ trapezoids, we have

$$\mathbb{E}[X]  =  \sum\limits_{\Delta \in Bad}p^{|D(\Delta)|} (1 - p)^{|I(\Delta)|}\leq \sum\limits_{1 \leq i \leq 4}20t{n \choose i} \left(\frac{r}{n}\right)^{i}\left( 1- \frac{r}{n}\right)^{n/2}.$$

\noindent For $r = C_3\log t$, where $C_3$ is a sufficiently large constant and $ t \geq 2$, we have

$$\mathbb{E}[X] \leq   80t r^4e^{-r/2} = 80t(C_3\log t)^4 t^{-C_3/2}\leq  \frac{1}{3}.$$

\noindent Hence

$$\mathbb{E}\left[\frac{1}{9C_2r^2t}|\mathcal{T}(S)| + X\right] \leq \frac{2}{3}.$$

By setting $C_1 = 9C_2(C_3)^2$, there exists a sample $S$ such that

$$|\mathcal{T}(S)| \leq 9C_2r^2t = 9C_2(C_3\log t)^2t = C_1t\log^2t$$

\noindent and $X = 0$.  By the pigeonhole principle, there exists a good trapezoid $\Delta \in \mathcal{T}(S)$ that contains at least $n/(C_1t\log^2 t)$ points from $P$, and at least $n/2$ curves from $F$ do not intersect the interior of $\Delta$.  This completes the proof.
\end{proof}

\begin{lemma}
\label{structure}
Given a family $R$ of $n$ red $t$-monotone curves and a family $B$ of $n$ blue $t$-monotone curves in the plane such that $R\cup B$ is simple, there exists a constant $C_4$, subsets $R'\subset R, B'\subset B$ of size $n/(C_4t\log^2 t)^4$ each, trapezoids $\Delta_{bl}, \Delta_{br}$, and regions $\Delta_{rl}, \Delta_{rr}\subset \mathbb{R}^2$, such that
\begin{enumerate}
\item the left endpoint of each blue curve in $B'$ lies inside $\Delta_{bl}$,
\item the right endpoint of each blue curve in $B'$ lies inside $\Delta_{br}$,
\item the left endpoint of each red curve in $R'$ lies inside $\Delta_{rl}$,
\item the right endpoint of each red curve in $R'$ lies inside $\Delta_{rr}$,
\item $\Delta_{bl}\cup\Delta_{br}$ is disjoint to $\Delta_{rl}\cup \Delta_{rr}$, and
\item no curve in $R'$ intersects $\Delta_{bl}\cup\Delta_{br}$, and no curve in $B'$ intersects $\Delta_{rl}\cup \Delta_{rr}$.
\end{enumerate}
\end{lemma}

\begin{proof} Let $P_{bl}$ be the set of left endpoints among the blue curves in $B$, and apply Lemma \ref{sample} to the red $t$-monotone curves $R$ and the point set $P_{bl}$.  Then we obtain subsets $P'_{bl} \subset P_{bl}$, $R_1\subset R$, and a trapezoid $\Delta_{bl}$, such that $|P'_{bl}|,|R_1| \geq n/(C_1t\log^2 t)$, $P'_{bl} \subset \Delta_{bl}$, and no curve in $R_1$ intersects the interior of $\Delta_{bl}$.  Let $B_1\subset B$ be the blue curves whose left endpoint belongs to $P'_{bl}$, and discard all curves not in $B_1\cup R_1$.  See Figure \ref{pbll}.

Let $P_{br}$ be the right endpoint of the curves in $B_1$, and apply Lemma \ref{sample} to $P_{br}$ and $R_1$.  Then again, we obtain subsets $P'_{br} \subset P_{br}$, $R_2\subset R_1$, and a trapezoid $\Delta_{br}$, such that $|P'_{br}|,|R_2| \geq n/(C_1t\log^2 t)^2$, and no curve in $R_2$ intersects the interior of $\Delta_{br}$.  Let $B_2\subset B_1$ be the blue curves whose right endpoint belongs to $P'_{br}$, and discard all curves not in $B_2\cup R_2$.  See Figure \ref{pbr}.

We repeat this entire process to the curves in $B_2$ with the endpoints of $R_2$, and obtain subsets $B_3\subset B_2$, $R_3\subset R_2$, trapezoids $\Delta_{bl},\Delta_{br},\Delta_3,\Delta_4$, such that

\begin{enumerate}
 \item $|B_3|,|R_3| \geq n/(C_1t\log^2 t)^4$,

\item the left endpoint of each blue curve in $B_3$ lies inside $\Delta_{bl}$,

\item the right endpoint of each blue curve in $B_3$ lies inside $\Delta_{br}$,

\item the left endpoint of each red curve in $R_3$ lies inside $\Delta_{3}$,

\item the right endpoint of each red curve in $R_3$ lies inside $\Delta_{4}$, and

\item no curve in $R_3$ intersects $\Delta_{bl}\cup\Delta_{br}$, and no curve in $B_3$ intersects $\Delta_{3}\cup \Delta_{4}$.
\end{enumerate}

 \begin{figure}
  \centering
  \subfigure[]{\label{pbll}\includegraphics[width=0.2\textwidth]{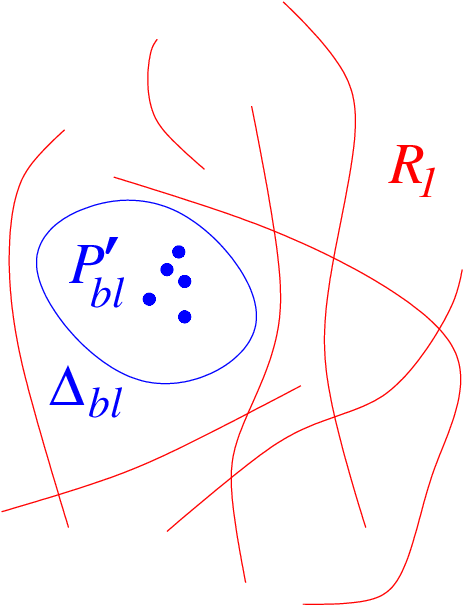}}\hspace{.5cm}
\subfigure[]{\label{pbr}\includegraphics[width=0.3\textwidth]{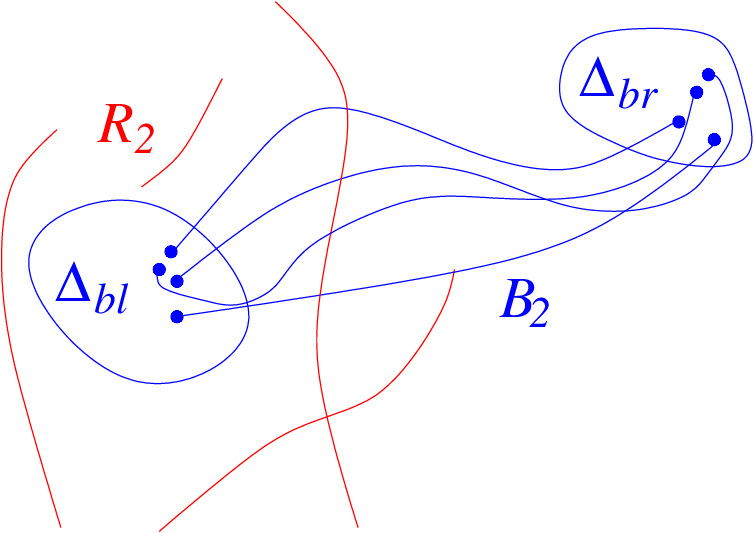}}\hspace{.5cm}
\subfigure[]{\label{br}\includegraphics[width=0.3\textwidth]{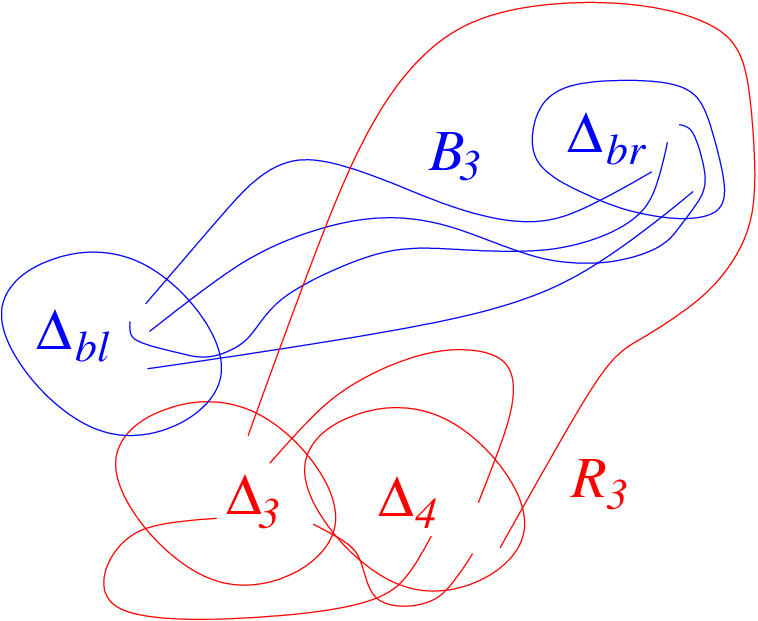}}
 \caption{Finding regions $\Delta_{bl},\Delta_{br},\Delta_3,\Delta_4$.}
  \label{case1}
\end{figure}

See Figure \ref{br}.  Now we will find regions $\Delta_{rl}\subset \Delta_3$ and $\Delta_{rr}\subset \Delta_4$ that satisfy the disjointness property (property 5). Let $P_{rl}$ be the set of left endpoints of $R_3$.  Recall that regions $\Delta_{bl},\Delta_{br},\Delta_3$, and $\Delta_4$ are trapezoids.  Since $R\cup B$ is simple,  this implies that

$$\mathbb{R}^2 \setminus (bd(\Delta_3)\cup bd(\Delta_{bl}\cup\Delta_{br}))$$

\noindent has at most (say) 80 connected components.  Since no point from $P_{rl}$ lies inside $\Delta_{bl}\cup\Delta_{br}$, by the pigeonhole principle there exists a simply connected region $\Delta_{rl}$ that contains at least $n/(3C_1t\log^2 t)^4$ points of $P_{rl}$, and $\Delta_{rl}$ is disjoint to $\Delta_{bl}\cup \Delta_{br}$.  Let $P'_{rl}\subset P_{rl}$ be the points that lie inside $\Delta_{rl}$, and let $R_4$ be the red curves whose left endpoints are in $P'_{rl}$.

We repeat this process to find region $\Delta_{rr}\subset \Delta_4$ and $R_5\subset R_4$, such that $|R_{5}| \geq n/(9C_1t\log^2 t)^4$, $\Delta_{rr}$ contains the right endpoints of $R_5$, and $\Delta_{rr}$ is disjoint to $\Delta_{bl}\cup\Delta_{br}$.  By letting $C_4 = 9C_1$, the statement of the lemma follows.

\end{proof}

We are now ready to prove Theorem \ref{key}.

\medskip

\noindent \emph{Proof of Theorem \ref{key}.}  We start by applying Lemma \ref{structure}, to obtain subsets $R_1\subset R$ and $B_1\subset B$, trapezoids $\Delta_{bl},\Delta_{br}$, and regions $\Delta_{rl},\Delta_{rr}$, with the properties described in Lemma \ref{structure}.  Suppose at least $|B_1|/2$ blue curves lie completely inside $\Delta_{bl}\cup\Delta_{br}$.  Then by property 6 in Lemma \ref{structure}, all of these blue curves must be disjoint to every red curve in $R_1$, and we are done.  Therefore, we can assume that there exist $B_2\subset B_1$ such that $|B_2| \geq |B_1|/2$, and each blue curve in $B_2$ does not lie completely inside $\Delta_{bl}\cup\Delta_{br}$.  Fix a curve $\alpha \in B_2$, and let $\alpha'$ be a subcurve of $\alpha$ that lies in $\mathbb{R}^2\setminus (\Delta_{bl}\cup \Delta_{br})$ and has endpoints on $bd(\Delta_{bl})$ and $bd(\Delta_{br})$.  See Figure \ref{alpha1}.  Now the proof falls into several cases.

\medskip

\noindent \emph{Case 1.}  Suppose that for at least $|B_2|/3$ curves $\gamma \in B_2$, regions $\Delta_{rl}$ and $\Delta_{rr}$ both lie in the same cell in the arrangement

$$\gamma\cup \alpha'\cup bd(\Delta_{bl}\cup\Delta_{br}).$$

\noindent See Figure \ref{c1}.  Then each red curve $\beta \in R_1$ intersects $\gamma$ if and only if $\beta$ intersects $\alpha'$.  Indeed, if $\beta$ intersects $\alpha'$, then $\beta$ must intersect $\gamma$ in order to come back inside the cell (since $R_1\cup B_1$ is simple and $\beta$ does not intersect $\Delta_{bl} \cup\Delta_{br}$).  Likewise, if $\beta$ is disjoint to $\alpha'$, then $\beta$ must lie completely inside of the cell.  Since at least half of the red curves in $R_1$ either intersect or are disjoint to $\alpha'$, the statement of the theorem follows.

 \begin{figure}[h]
  \centering
    \subfigure[Drawing of $\alpha'$.  Note that $\Delta_{bl}$ and $\Delta_{br}$ may or may not be disjoint.]{\label{alpha1}\includegraphics[width=0.3\textwidth]{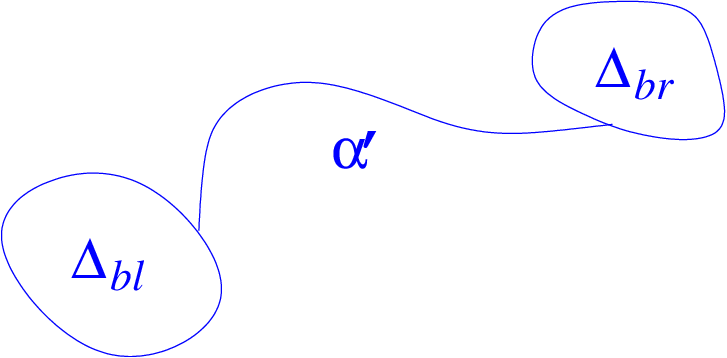}}\hspace{.5cm}
  \subfigure[Both $\Delta_{rl}$ and $\Delta_{rr}$ lie in the same cell.]{\label{c1}\includegraphics[width=0.3\textwidth]{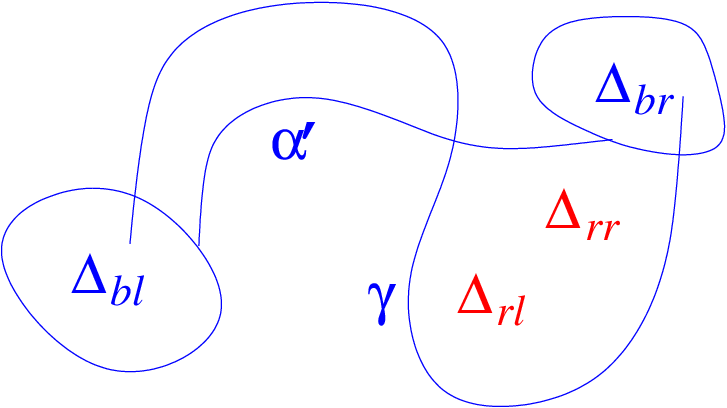}}\hspace{.5cm}
  \subfigure[Here, $\Delta_{rl}$ lies in a cell that is not incident to $\alpha'$.]{\label{c2}\includegraphics[width=0.3\textwidth]{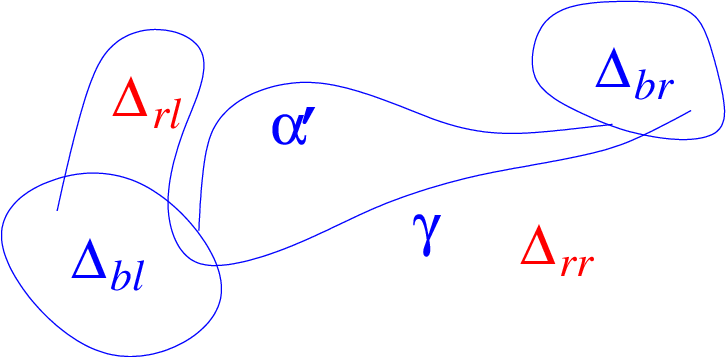}}
 \caption{Drawing of $\alpha'$, cases 1 and 2.}
  \label{case1}
\end{figure}

\medskip

\noindent \emph{Case 2.}  Suppose that for at least $|B_2|/3$ curves $\gamma \in B_2$, regions $\Delta_{rl}$ and $\Delta_{rr}$ lie in distinct cells in the arrangement

$$\gamma\cup \alpha'\cup bd(\Delta_{bl}\cup\Delta_{br}),$$

  \noindent and one of these cells is not incident to $\alpha'$ (i.e., the cell is surrounded by $\gamma$ and $\Delta_{bl}\cup\Delta_{br}$).  See Figure \ref{c2}.  Then clearly every red edge in $R_1$ must intersect $\gamma$, and the statement follows.

\medskip

 Therefore, we can assume that we are not in case 1 or 2.  Hence there exists a subset $B_3 \subset B_2$ such that $|B_3| \geq |B_2|/3$, and for each curve $\gamma \in B_2$, regions $\Delta_{rl}$ and $\Delta_{rr}$ lie in distinct cells in the arrangement

  $$\gamma\cup \alpha'\cup bd(\Delta_{bl}\cup\Delta_{br}),$$

  \noindent and both of these cells are incident to $\alpha'$.

  \medskip

\noindent \emph{Case 3.}  Suppose that for at least $|B_3|/3$ curves $\gamma \in B_3$, in the arrangement

  $$\gamma\cup \alpha'\cup bd(\Delta_{bl}\cup\Delta_{br}),$$

\noindent $\Delta_{rl}$ and $\Delta_{rr}$ lie in distinct cells that share $\alpha''\subset\alpha'$ as a common side. See Figures \ref{c33} and \ref{c4a}.

 \begin{figure}[h]
  \centering
    \subfigure[Subcurve $\alpha''$ drawn thick.]{\label{c33}\includegraphics[width=0.3\textwidth]{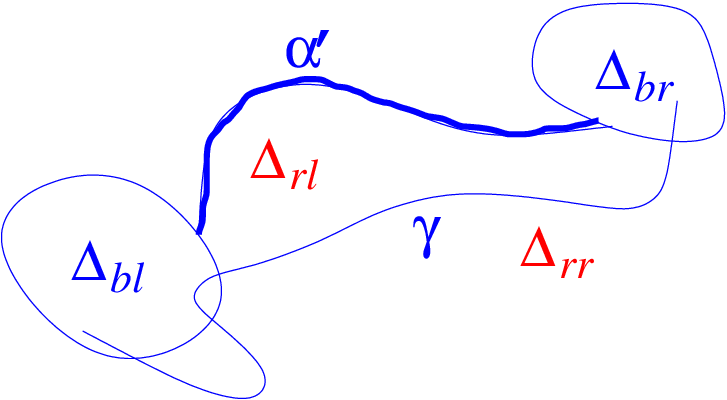}}\hspace{.5cm}
  \subfigure[Subcurve $\alpha''$ drawn thick.]{\label{c4a}\includegraphics[width=0.3\textwidth]{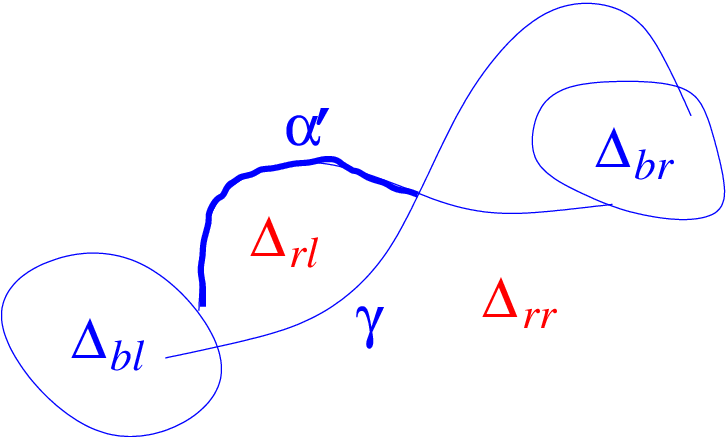}}
 \caption{Defining $\alpha''$.}
  \label{case1}
\end{figure}

Notice that $\alpha''$ is unique since $R\cup B$ is simple.  Then each red curve $\beta \in R_1$ intersects $\gamma$ if and only if $\beta$ is disjoint to $\alpha'$.  Indeed, notice that one of these cells must be surrounded by $\alpha''$, $\gamma$, $\Delta_{bl}\cup\Delta_{br}$.  Without loss of generality, assume that $\Delta_{rl}$ lies in such a cell.  Starting at $\Delta_{rl}$, each curve $\beta \in R_1$ must intersect either $\gamma$ or $\alpha''$ in order to leave the cell.  Suppose $\beta$ first intersects $\alpha''$.  Then $\beta$ must now be in the cell that $\Delta_{rr}$ lies in since these cells share $\alpha''$ as a common side. Since $R\cup B$ is simple, $\beta$ must remain the the current cell and, therefore, must be disjoint to $\gamma$.

Now if $\beta$ crossed $\gamma$ first, then $\beta$ must be in the same cell as $\Delta_{rr}$.  Indeed, otherwise $\beta$ would then cross $\alpha'$ and either be in the same cell as $\Delta_{rl}$ is in, or would be in a cell that is not adjacent to $\alpha''$.   Since $R\cup B$ is simple, we have a contradiction in either case. Hence $\beta$ is disjoint to $\alpha'$.

Since at least of half of the red curves in $R_1$ either intersect or are disjoint to $\alpha'$, the statement of the theorem follows.

\medskip

\noindent \emph{Case 4.}  Suppose that for at least $|B_3|/3$ curves $\gamma \in B_3$, in the arrangement

  $$\gamma\cup \alpha'\cup bd(\Delta_{bl}\cup\Delta_{br}),$$

\noindent $\Delta_{rl}$ and $\Delta_{rr}$ lie in distinct adjacent cells not sharing $\alpha'$ as a common side.  See in Figure \ref{c4c2}.  Then clearly, each red curve in $R_1$ intersects $\gamma$ since $R\cup B$ is simple, and the statement of the theorem follows.

\medskip

\noindent \emph{Case 5.}    Suppose that for at least $|B_3|/3$ curves $\gamma \in B_2$, in the arrangement

  $$\gamma\cup \alpha'\cup bd(\Delta_{bl}\cup\Delta_{br}),$$

\noindent $\Delta_{rl}$ and $\Delta_{rr}$ lie in distinct non-adjacent cells as in Figure \ref{c4b}.  This is the final case.  Then, clearly, each red curve in $R_1$ intersects $\gamma$ since $R\cup B$ is simple, and the statement of the theorem follows.$\hfill\square$

 \begin{figure}
  \centering
    \subfigure[Case 4.]{\label{c4c2}\includegraphics[width=0.2\textwidth]{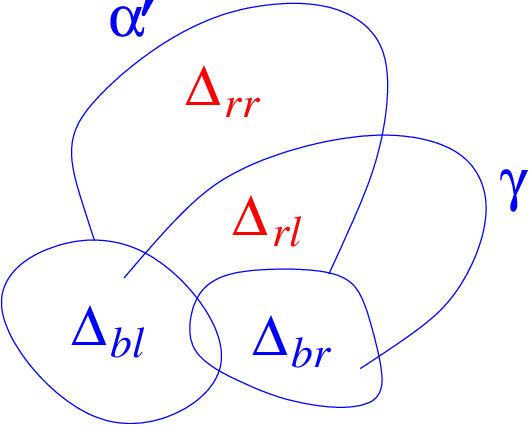}}\hspace{1.5cm}
\subfigure[Case 5.]{\label{c4b}\includegraphics[width=0.3\textwidth]{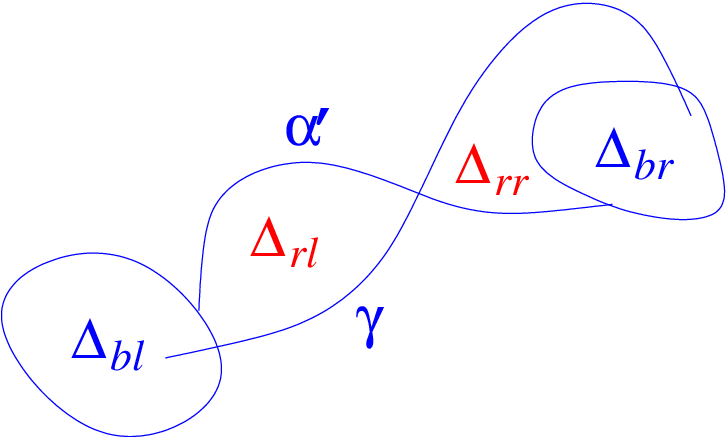}}
 \caption{Last two remaining cases.}
  \label{case1}
\end{figure}

\section{Proof of Theorems \ref{main1} and \ref{main2}}

The proof of Theorems \ref{main1} and \ref{main2} is now a standard application of Szemer\'edi's regularity lemma (see \cite{mat}, \cite{solymosi}).   Let us recall the weak bipartite regularity lemma.

\begin{theorem}[\cite{mat}]
\label{sz}
Let $G = (X_1,X_2,E)$ be a bipartite graph with parts $X_1$ and $X_2$ such that $|X_1| = |X_2| = n$.  Let $c\geq 2$.  If $|E| \geq \epsilon n^2$, then there exists subsets $Y_1\subset X_1, Y_2\subset X_2$ such that

\begin{enumerate}

\item  $|Y_1| = |Y_2| = \epsilon^{c^2} n$, and

\item $|E(Y_1,Y_2)| \geq \epsilon |Y_1||Y_2|$, and

\item $|E(Z_1,Z_2)| > 0$ for any $Z_i\subset Y_i$, with $|Z_i|\geq |Y_i|/c$ for $i \in \{1,2\}$.

\end{enumerate}

\end{theorem}

\medskip

\noindent \emph{Proof of Theorem \ref{main1}.}  Let $F$ be a simple family of $n$ $t$-monotone curves with $\epsilon n^2$ intersecting pairs.  Then by a random partition, we can partition $F$ into two subfamilies $F_1$, $F_2$, such that $|F_1|,|F_2| \geq n/3$ and the number of pairs of curves, one from $F_1$ and one from $F_2$, that intersect is at least $\epsilon n^2/2$.

Set $c = c''_t$, where $c''_t$ is defined in Theorem \ref{key}.  Then by Theorem \ref{sz}, there exist subsets $Y_1 \subset F_1$, $Y_2 \subset F_2$, of size $(\epsilon/2)^{c^2} n/3$ each, such that for any subsets $Z_1\subset Y_1$, $Z_2\subset Y_2$ with $|Z_i| \geq   |Y_i|/c$ for $i \in \{1,2\}$, there must be a curve in $Z_1$ that intersects a curve in $Z_2$.  By Theorem \ref{key}, there exists subsets $Z_1\subset Y_1,Z_2\subset Y_2$ such that every curve in $Z_1$ intersects every curve in $Z_2$, and

$$|Z_i| \geq \frac{|Y_i|}{c} \geq \frac{(\epsilon/2)^{c^2} }{3c}n \geq   \epsilon^{c_t}n,$$

\noindent where $c_t$ depends only on $t$. $\hfill\square$

\medskip

\noindent Theorem \ref{main2} follows by replacing the word ``intersect'' with ``disjoint'' in the proof above.

\section{Simple topological graphs with no $k$ pairwise disjoint edges}

As defined in \cite{crossing}, the {\it odd-crossing number} $\ocn(G)$ of a graph $G$ is the minimum possible number of unordered pairs of edges that cross an odd number of times over all drawings of $G$. The {\it bisection width} of a graph $G$, denoted by $b(G)$, is the smallest nonnegative integer such that there is a partition of
 the vertex set $V=V_1 \, \dot{\cup} \, V_2$ with $\frac{1}{3}\cdot |V|\leq |V_i|\leq \frac{2}{3}\cdot |V|$ for $i=1,2$, and  $|E(V_1,V_2)|= b(G)$. The following lemma, due to Pach and T\'oth, relates the odd-crossing number of a graph to its bisection width.

\begin{lemma}[\cite{pach2}]\label{bisect}
There is an absolute constant $c_1$ such that if $G$ is a graph with
 $n$ vertices of degrees $d_1,\ldots,d_n$, then
 $$b(G)\leq c_1\log n \sqrt{\ocn(G)+\sum_{i=1}^n d_i^2} .$$
 \end{lemma}

\medskip

 \noindent Since all graphs contain a bipartite subgraph with at least half of its edges, Theorem~\ref{disjoint} immediately follows from the following theorem.

 \begin{theorem}
 \label{proof}

 Let $G = (V,E)$ be an $n$-vertex simple topological bipartite graph with edges drawn as $t$-monotone curves.  If $G$ does not contain $k$ pairwise disjoint edges, then $|E(G)| \leq n(\log n)^{c'_t\log k}$, where $c'_t$ is a constant that depends only on $t$.

 \end{theorem}

\begin{proof} We define $f(n,k)$ to be the maximum number of edges in an $n$-vertex simple topological bipartite graph with edges drawn as $t$-monotone curves, that does not contain $k$ pairwise disjoint edges.  We will prove by induction on $n$ and $k$ that

 $$f(n,k) \leq  n(\log n)^{c'_t\log k}.$$

 \noindent Note that $f(n,k) \leq {n\choose 2}$, and by \cite{fulek} $f(n,2) \leq 1.43n$.  Now assume the statement is true for $n' < n$ and $k' < k$, and let $G$ be an $n$-vertex simple topological bipartite graph with edges drawn as $t$-monotone curves, which does not contain $k$ pairwise disjoint edges. The proof falls into two cases.

 \medskip

 \noindent \emph{Case 1.}  Suppose there are at least $|E(G)|^2/((2c_1)^2\log^6 n)$ disjoint pairs of edges in $G$.  By Theorem \ref{main2}, there exist subsets $E_1,E_2 \subset E(G)$ such that $|E_1|,|E_2| \geq |E(G)|/(c\log n)^{6c_t},$ every edge in $E_1$ is disjoint to every edge in $E_2$, and $c$ is an absolute constant.  Since $G$ does not contain $k$ pairwise disjoint edges, this implies that there exists an $i \in \{1,2\}$ such that $|E_i|$ does not contain $k/2$ pairwise disjoint edges.  Hence

 $$\frac{|E(G)|}{(c\log n)^{6c_t}} \leq |E_i| \leq f(n,k/2).$$

 \noindent By the induction hypothesis, we have

 $$f(n,k/2) \leq n(\log n)^{c'_t\log(k/2)}  \leq n(\log n)^{c'_t\log(k) - c'_t}.$$

 \noindent For sufficiently large $c'_t$, we have $|E(G)|  \leq n(\log n)^{c'_t\log(k)}.$

\medskip

\noindent \emph{Case 2.}  Suppose there are at most $|E(G)|^2/((2c_1)^2\log^6n)$ disjoint pairs of edges in $G$.  In what follows, we will apply a redrawing technique that was used by Pach and T\'oth \cite{pach2}.  Since $G$ is bipartite, let $V_a$ and $V_b$ be its vertex class.  By applying a suitable homeomorphism to the plane, we can redraw $G$ such that

 \begin{enumerate}

 \item the vertices in $V_a$ are above the line $y = 1$, the vertices in $V_b$ are below the line $y = 0$,

 \item edges in the strip $0 \leq y \leq 1$ are vertical segments,

   \item we have neither created nor removed any crossings.
 \end{enumerate}

   \noindent  Now we reflect the part of $G$ that lies above the $y = 1$ line about the $y$-axis.
 Then erase the edges in the strip $0\leq y \leq 1$ and replace them by straight line segments that
reconnect the corresponding pairs on the line $y = 0$ and $y = 1$. Note that our topological graph is no longer simple, and the edges are no longer $t$-monotone.  See Figure \ref{redraw}, and note that our graph is no longer simple.

\begin{figure}
\centering
\includegraphics[width=.7\textwidth]{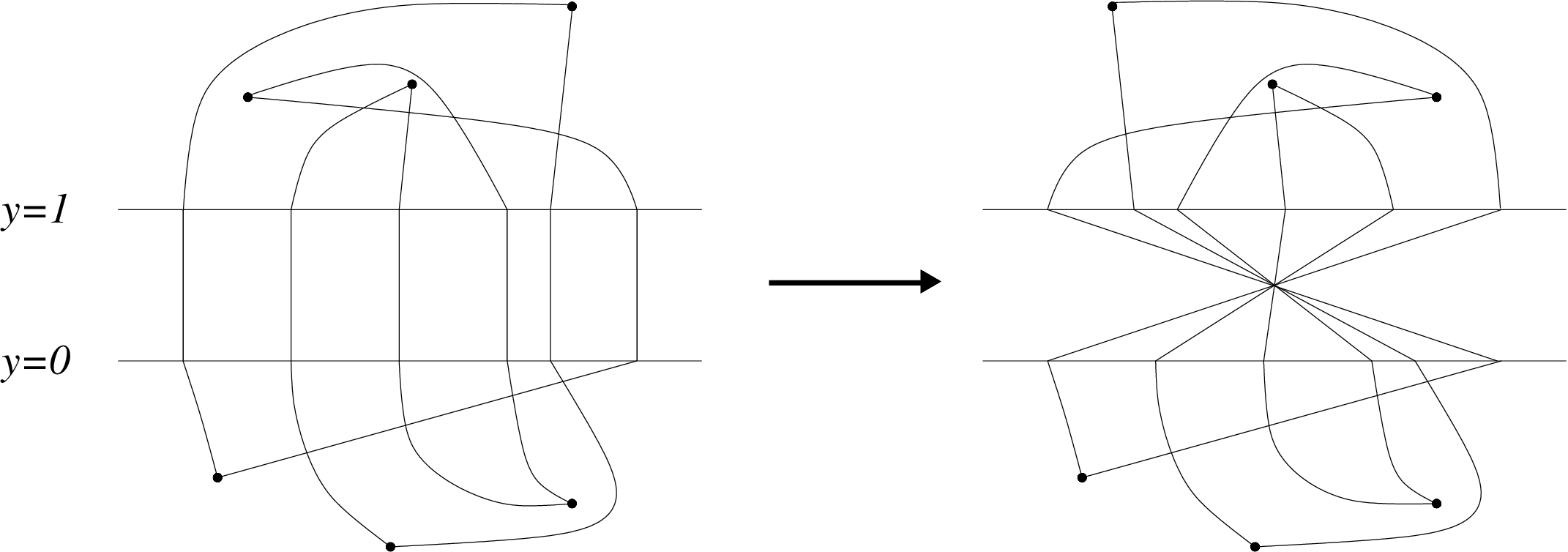}
\caption{Redrawing procedure}
\label{redraw}
\end{figure}

 Notice that if any two edges crossed in the original drawing, then they must cross an even number of times in the new drawing.  Indeed, suppose the edges $e_1$ and $e_2$ crossed in the original drawing.  Since $G$ is simple, they share exactly 1 point in common.  Let $k_i$ denote the number of times edge $e_i$ crosses the strip for $i \in \{1,2\}$, and note that $k_i$ must be odd. After we have redrawn our graph, these $k_1+k_2$ segments inside the strip will now pairwise cross, creating ${k_1 + k_2 \choose 2}$ crossing points.  Since edge $e_i$ will now cross itself ${k_i \choose 2}$ times, this implies that there are now

 \begin{equation}
 \label{parity}
 {k_1 + k_2 \choose 2} - {k_1\choose 2} - {k_2 \choose 2}
 \end{equation}

 \noindent crossing points between edges $e_1$ and $e_2$ inside the strip.  One can easily check that (\ref{parity}) is odd when $k_1$ and $k_2$ are odd.  Since $e_1$ and $e_2$ had 1 point in common outside of the strip, this implies that $e_1$ and $e_2$ cross each other an even number of times. Note that one can easily get rid of self-intersections by making local modifications in a small neighborhood at these crossing points.

 Hence, the odd-crossing number in our new drawing is at most the number of disjoint pair of edges in the original drawing of $G$, plus the number of pair of edges that share a common vertex.  Since there are at most
 $$\sum\limits_{v\in V(G)} d^2(v) \leq 2|E(G)|n$$

 \noindent pairs of edges that share a vertex in $G$, this implies

$$\ocn(G) \leq \frac{|E(G)|^2}{(2c_1)^2\log^6n} + 2|E(G)|n.$$

\noindent   By Lemma~\ref{bisect}, there is a partition of the vertex set $V=V_1 \, \dot{\cup} \, V_2$ with $\frac{1}{3}\cdot
 |V|\leq |V_i|\leq \frac{2}{3}\cdot |V|$ for $i=1,2$ and

 $$|E(V_1,V_2)| \leq b(G) \leq c_1\log n \sqrt{\frac{|E(G)|^2}{(2c_1)^2\log^6 n} +  4n|E(G)|  }.$$

 \noindent  If $|E(G)|^2/((2c_1)^2\log^6 n) \leq  4n|E(G)|$, then we have $|E(G)| \leq n(\log n)^{c'_t\log k}$ and we are done.  Therefore, we can assume

   $$ b(G) \leq c_1\log n \sqrt{\frac{2|E(G)|^2}{(2c_1)^2\log^6 n} } \leq \frac{ |E(G)|}{\log^2 n}.$$

\noindent Let $|V_1| = n_1$ and $|V_2| = n_2$.  By the induction hypothesis we have

 $$
 \begin{array}{ccl}
 |E(G)| & \leq & b(G)+  n_1(\log n_1)^{c_t'\log k} + n_2(\log n_2)^{c_t'\log k}\\\\
     & \leq &  \frac{|E(G)|}{ \log^2  n}  + n(\log(2n/3))^{c_t'\log k} \\\\
    & \leq &   \frac{|E(G)|}{ \log^2  n}+  n(\log n  - \log(3/2))^{c_t'\log k},  \\\\
 \end{array}$$

\noindent which implies
$$|E(G)| \leq n(\log n)^{c_t'\log k} \frac{(1 - \log(3/2)/\log n)^{c_t'\log k}  }{1 - 1/\log^2n}  \leq n(\log n)^{c_t'\log k}.$$

\end{proof}

\section{Concluding Remarks}

  It would be interesting to see if one can remove the $t$-monotone condition in Theorem \ref{main2} and Lemma \ref{sample}.

\begin{problem}

Given a simple family $F$ of $n$ curves in the plane, and an $n$-element point set $P$ such that no point of $P$ lies on a curve from $F$, does there exist a constant $\epsilon > 0$, a region $\Delta\subset \mathbb{R}^2$, subsets $F'\subset F, P'\subset P$ of size $\epsilon n$ each, such that $P'\subset \Delta$ and every curve in $F'$ does not intersect the interior of $\Delta$?

\end{problem}

We list two more unsolved problems related to this paper.

\begin{problem}
Let $G$ be an $n$-vertex simple topological graph with edges drawn as 2-monotone curves.  If $G$ has no two disjoint edges, then does $G$ have at most $n$ edges? What if the edges are drawn as 3-monotone curves?

\end{problem}

\begin{problem}

Given a simple family $F$ of $n$ 2-monotone curves in the plane with no $3$ pairwise disjoint members, can one color the members in $F$ with at most $c$ colors, such that each color class consists of pairwise crossing members?

\end{problem}

\end{document}